\documentclass[12pt,journal]{IEEEtran}
\onecolumn
\usepackage{mathrsfs}
\usepackage{algorithmic}
\usepackage{algorithm}
\usepackage{amsmath}
\usepackage{amsthm}
\usepackage{graphicx}
\usepackage{amssymb}
\usepackage{epstopdf}
\usepackage{enumerate}
\usepackage{longtable,tabularx,float}
\usepackage{cite}
\usepackage{mathcomp}
\usepackage{multirow}
\usepackage{supertabular}
\usepackage{stmaryrd}
\usepackage{color}
\usepackage{url}
\usepackage{makecell}
\usepackage[OT2,OT1]{fontenc}
\usepackage{bm}
\usepackage{bbm}
\usepackage{caption}
\usepackage{soul}
\newtheorem{corollary}{Corollary}[section]
\newtheorem{definition}{Definition}[section]
\newtheorem{lemma}{Lemma}[section]

\newtheorem{theorem}{Theorem}[section]
\newenvironment{thmbis}[1]
 {\renewcommand{\thetheorem}{\ref{#1}}%
 \addtocounter{theorem}{-1}%
 \begin{theorem}}
 {\end{theorem}}
\newtheorem{problem}{Problem}[section]

\newtheorem{example}{Example}[section]
\newtheorem{construction}{Construction}[section]

\newtheorem{claim}{Claim}[section]

\newtheorem{remark}{Remark}[section]

\newtheorem{question}{Question}[section]

\setstcolor{red}

 \usepackage{xcolor}
\definecolor{cream}{RGB}{203, 237, 204}

\begin{document}

\title{Separating hash families with large universe}

\author{Xin~Wei,~Xiande~Zhang,~and~Gennian~Ge
        \thanks{X. Wei ({\tt weixinma@mail.ustc.edu.cn}) is with the School of Mathematical Sciences, University of Science and Technology of China, Hefei, 230026, Anhui, China.}

\thanks{X. Zhang ({\tt drzhangx@ustc.edu.cn}) is with the School of Mathematical Sciences,
University of Science and Technology of China, Hefei, 230026, and with Hefei National Laboratory,  Hefei, 230088, Anhui, China.  The research of X. Zhang is supported by the NSFC
under Grants No. 12171452 and No. 12231014, the Innovation Program for Quantum Science and Technology
(2021ZD0302902) and the National Key Research and
Development Program of China (2020YFA0713100).}

\thanks{G. Ge ({\tt gnge@zju.edu.cn}) is with the School of Mathematical Sciences, Capital Normal University, Beijing, 100048, China. The research of G. Ge was supported by the National Key Research and Development Program of China under Grant 2020YFA0712100 and Grant 2018YFA0704703, the National Natural Science Foundation of China under Grant 11971325 and Grant 12231014, and Beijing Scholars Program.
}
\thanks{
Declarations of interest: none.}
}
\maketitle

\begin{abstract}
 Separating hash families are useful combinatorial structures which generalize
several well-studied objects in cryptography and coding theory. Let $p_t(N, q)$  denote the maximum size of universe for a $t$-perfect hash
family of length $N$ over an alphabet of size $q$. In this paper, we show that $q^{2-o(1)}<p_t(t, q)=o(q^2)$ for all $t\geq 3$, which answers an open problem about separating hash families raised by Blackburn et al. in 2008 for certain parameters.  Previously, this result was known only for $t=3, 4$. Our proof is obtained by establishing the existence of a large set of integers avoiding nontrivial solutions to a set of correlated linear equations.

\end{abstract}

\begin{IEEEkeywords}
\boldmath separating hash family, perfect hash family, solution-free set.
\end{IEEEkeywords}

\section{Introduction}\label{sec_introduction}

Let $X$ and $Y$ be two sets with sizes $n$ and $q$, respectively. An $(N; n, q)$-{\it hash family} is a family $\mathcal F$ of $N$ functions from $X$ to $Y$, while the sets $X$ and $Y$ are called the {\it universe} and the {\it alphabet} of $\mathcal F$, respectively. For some pairwise disjoint subsets $C_1, \ldots, C_t\subset X$, we say a function $f\in \mathcal F$ {\it separates} $C_1, \ldots, C_t$ if $f(C_1), \ldots, f(C_t)$ are pairwise disjoint. An $(N; n, q)$-hash family  $\mathcal F$ is called an $(N; n, q, \{w_1, \ldots, w_t\})$-\emph{separating hash family}, denoted by $SHF(N; n, q, \{w_1, \ldots, w_t\})$, if for all pairwise disjoint subsets $C_1, \ldots, C_t\subset X$ with $|C_i|=w_i\geq 1$ for $1\le i\le t$, there exists at least one function $f\in\mathcal F$ that separates $C_1, \ldots, C_t$. We call the multiset $\{w_1, \ldots, w_t\}$ the {\it type}  of the separating hash family $\mathcal F$. Note that
an $(N; n, q)$-hash family can be depicted as an $N\times n$ array $\mathcal A$ (\emph{matrix representation}) with rows indexed by $\mathcal F$, columns indexed by $X$ and with elements from $Y$, or can be viewed as a $q$-ary code of length $N$ and size $n$ by collecting all columns in $\mathcal A$ (\emph{code representation}).

The notion of separating hash family (SHF) was first introduced  by Stinson et al. in~\cite{stinson2008generalized} as a generalization of several well-studied classes of combinatorial objects, such as frameproof codes~\cite{boneh1998collusion, stinson2000secure} and parent-identifying codes~\cite{hollmann1998codes}.
Especially when $w_1= w_2= \cdots =w_t=1$, an $SHF(N; n, q, \{1, \ldots, 1\})$ is known as a $t$-perfect hash family (PHF), and denoted by $PHF(N; n, q, t)$.  PHF was first introduced by Mehlhorn~\cite{mehlhorn2013data} in 1984 and has applications in cryptography~\cite{stinson2000new, walker2007perfect}, database management~\cite{mehlhorn2013data}, designs of circuits~\cite{newman1995lower} and algorithms~\cite{alon1996derandomization}. Further when $q=t$, a $PHF(N; n, q, q)$ is known as a $q$-hash code \cite{fredman1984size,korner1986fredman}, the code rate of which has been  studied recently \cite{guruswami2022beating,xing2023beating}.

 Given integers $N$, $q$ and $w_1, \ldots, w_t$, denote the maximum size of the universe $n$ as $C(N, q, \{w_1, \ldots, w_t\})$, such that there exists an $SHF(N; n, q, \{w_1, \ldots, w_t\})$.  The study of the value of  $C(N, q, \{w_1, \ldots, w_t\})$ is a fundamental problem in combinatorics, information theory, and computer science. Many efforts have been made to determine the bounds of it for different cases, see for example, \cite{bazrafshan2011bounds, blackburn2008bound, shangguan2016separating, ge2019some}.

When $N$ and $w_1, \ldots, w_t$ are fixed, $C(N, q, \{w_1, \ldots, w_t\})$ can be regarded as a function of $q$. Let $u:=\sum_{i=1}^t w_i$ (Throughout the paper $u$ always refers to this formula). A general upper bound \[C(N, q, \{w_1, \ldots, w_t\})\leq \gamma q^{\left\lceil \frac{N}{u-1}\right\rceil}\] has been obtained by Blackburn et al. in \cite{blackburn2008bound}, where $\gamma=w_1w_2+u-w_1-w_2$ with $w_1,w_2\leq w_i$ for $3\leq i\leq t$.  The constant $\gamma$  has been improved  by several authors, see \cite{bazrafshan2011bounds,shangguan2016separating}.
As for the lower bound, Blackburn~\cite{blackburn2000perfect} used the Lov\'{a}sz local lemma and gave a probabilistic construction, which showed that
$$C(N, q, \{w_1, \ldots, w_t\})> \frac 1 {2^u}\left(\frac{q}{\binom u 2}\right)^{\frac N {u-1}}=\Omega_{u, N}(q^{\frac N {u-1}}).$$
This means that when $(u-1)\mid N$, the exponent $\lceil N\slash(u-1)\rceil$ in the upper bound of $C(N, q, \{w_1, \ldots, w_t\})$ matches the lower bound. Especially for $N=u-1$, it has been proved that $C(u-1, q, \{w_1, \ldots, w_t\})=(u-1)q-o(q)$ \cite{bazrafshan2011bounds}. However, it is still an open problem to determine whether the exponent is tight when $(u-1)\nmid N$.
\begin{question}\label{question_blackburn}
(See \cite{blackburn2008bound}.)
Let $N$ and $w_i$ be fixed positive integers. If $(u-1)\nmid N$, then for sufficiently large $q$ and arbitrarily small $\epsilon>0$, does there exist an $SHF(N; n, q, \{w_1, \ldots, w_t\})$ such that $n\ge q^{\lceil N\slash (u-1)\rceil-\epsilon}$?
\end{question}

Note that  the general upper and lower bounds of SHFs above are similar to those for the classical hypergraph Tur\'{a}n problem introduced by Brown, Erd\H{o}s and S\'{o}s \cite{brown1973some} in the early
1970s. Let $f_r (n, v, e)$ denote the maximum number of edges in an $r$-uniform hypergraph on $n$ vertices, in which the union of any $e$ distinct edges contains at least $v + 1$ vertices. It was shown \cite{brown1973some} in general that $\Omega(n^{\frac{er-v}{e-1}})=f_r(n,v,e)=O(n^{\lceil\frac{er-v}{e-1}\rceil})$, and conjectured that  $n^{k-o(1)}<f_r(n,v,e)=o(n^k)$ holds when $v=er-(e-1)k+1$ for all fixed integers $r>k\geq 2$ and $e\geq 3$. Several sporadic cases were proved to be true, such as when $e=3$ \cite{alon2006extremal} or when $r=3$ \cite{solymosi2017small}, including the famous $(6,3)$-theorem of Ruzsa
and Szemer\'{e}di~\cite{ruzsa1978triple}.

Denote $p_t(N, q):=C(N, q, \{1, \ldots, 1\})$, in which there are $t$ copies of $1$ in its type. The first breakthrough  regarding to Question~\ref{question_blackburn} was obtained in~\cite{shangguan2016separating} for the PHFs when $N=u=t$,  which also answered an open problem of~\cite{walker2007perfect}.  In fact, they obtained the following result which is similar to the $(6,3)$-theorem.


\begin{theorem}\label{thm_old_PHF_bounds}
(See \cite{shangguan2016separating}.)
For large enough $q$, $q^{2-o(1)}<p_t(t, q)=o(q^2)$ for both $t= 3, 4$.
\end{theorem}
The upper bound follows from a combination of  a Johnson-type bound~\cite{shangguan2016separating} and the $(6,3)$-theorem by Ruzsa and Szemer\'{e}di~\cite{ruzsa1978triple}, and can be  generalized as follows.
\begin{theorem}\label{thm_general_upper}
(See \cite{ge2019some}.)
 Let $w_1, \ldots, w_t$ be fixed positive integers such that either $t\ne 2$ or $\min\{w_1, w_2\}\ne 1$. Then the existence of an $SHF(u; n, q, \{w_1, \ldots, w_t\})$ implies that $n=o(q^2)$ for
 sufficiently large $q$.
\end{theorem}


%

In this paper, we establish the lower bound $q^{2-o(1)}$ for all parameters listed in Theorem~\ref{thm_general_upper}, which gives a positive answer to Question~\ref{question_blackburn} for a wider range of type when $N=u$ (actually it is $N\in [u, 2u-3]$, see the conclusion). In particular, this result generalizes Theorem~\ref{thm_old_PHF_bounds} to all $t\geq 3$. Our formal statement is given below.

\begin{theorem}\label{thm_all}
Let $w_1, \ldots, w_t$ be fixed positive integers such that either $t\ne 2$ or $\min\{w_1, w_2\}\ne 1$. Then for sufficiently large $q$ we have $q^{2-o(1)}< C(\sum_{i=1}^t w_i, q, \{w_1, \ldots, w_t\})=o(q^2)$. In particular, $q^{2-o(1)}<p_t(t, q)=o(q^2)$ for all $t\geq 3$.
\end{theorem}
We remark that the assumption of the type in Theorem~\ref{thm_all} is necessary when $N=u$. For $t=2$ and $\{w_1, w_2\}=\{1, w\}$ for some $w$, an $SHF(N; n, q, \{1, w\})$ is known as a  $w$-frameproof code~\cite{blackburn2003frameproof}. One can use Reed-Solomon codes to show $C(N, q, \{1, w\})=\Omega (q^{\lceil\frac N w\rceil})$  \cite[Construction 2]{blackburn2003frameproof}, and consequently $C(u, q, \{1, w\})=\Omega (q^{2})$.

The lower bound of Theorem~\ref{thm_all} is obtained by explicit constructions of PHFs, which utilize a new type of solution-free-set problem for a set of linear equations defined in~\cite{shangguan2016separating}. Roughly speaking, the problem is to find a set $R=\{b_1,b_2,\ldots,b_t\}$ of $t$ distinct integers in $\{0,1,\ldots,q-1\}$ with small rank $r$\footnote{{The} rank is the difference between the maximum and minimum integers in $R$.}, such that there exists a big set $M\subset\{0,1,\ldots, \lfloor(q-1)/r\rfloor\}$ that contains no nontrivial solutions to any linear equation of the form
\[(u_1-u_k)x_1+(u_2-u_1)x_2+\cdots+(u_k-u_{k-1})x_k=0\] for all $k$-permutation sequences\footnote{ A $k$-permutation sequence means a sequence of $k$ mutually different elements.}
$(u_1,u_2,\ldots,u_k)$ of $R$ and for any $3\le k\leq t$. Such a set $M$ is called an \emph{$R$-solution-free} set.  From $R$ and $M$, one can construct a $PHF(t; q|M|, q, t)$ with columns $(y+b_1m,y+b_2m,\ldots,y+b_tm)^\top$, 
 $y\in \mathbb{{Z}}_q$ and $m\in M$, see \cite[Lemmas 6.1-6.2]{shangguan2016separating}. To get more columns for the PHF, one needs a set $M$ of big size and necessarily a set $R$ with small rank. In this paper, the main technical result regarding to this problem is described as follows.

\begin{theorem}\label{thm_sl} For $t\geq 3$ and sufficiently large $q$, there exists a set $R\subset \{0,1,\ldots,q-1\}$ of size $t$ and rank $r=q^{o(1)}$, which has an $R$-solution-free set in $\{0, 1, \ldots, \lfloor(q-1)\slash r\rfloor\}$ with size $q^{1-o(1)}$.
\end{theorem}

By Theorem~\ref{thm_sl} and the above construction, one can easily obtain $p_t(t, q)>q^{2-o(1)}$, which is the main part of Theorem~\ref{thm_all}. Previously, the result in Theorem~\ref{thm_sl} was known only for $t=3,4$ in \cite{shangguan2016separating}, which yields the lower bound in Theorem~\ref{thm_old_PHF_bounds}.


The paper is organized as follows. In Section~\ref{sec_preliminaries} we recall the relationship between the separating hash families and a solution-free-set problem, and then turn the proof of Theorem~\ref{thm_sl} to the existence problem of a large solution-free set of a certain kind of equations (Theorem~\ref{thm_trans_main_thm}). Some useful notations and an outline for the proof of Theorem~\ref{thm_trans_main_thm} are given in Section~\ref{sec_notations}. Sections~\ref{sec_link_lemma} and~\ref{sec_sprine_lemma} are devoted to prove Theorem~\ref{thm_trans_main_thm} by following the outline. A conclusion and some open problems are listed in Section~\ref{sec_conclusion}.


\section{Preliminaries and Main Results}\label{sec_preliminaries}
 For convenience, we use $[a, b]$ to denote the set of integers $\{a,a+1,\ldots,b\}$ for two integers $a\leq b$, and use $[n]$ to denote $[1, n]$ for short.  The logarithms are always under base $2$ by default.



We first define solution-free sets of  equations formally. Given a linear equation $\mathcal L$ with nonzero integer coefficients $(a_i)_{1\le i\le k}$ and an integer $b$
\begin{equation*}
a_1x_1+ \cdots+ a_kx_k= b,
\end{equation*}
an integer set $M$ is called {\it solution-free} of $\mathcal L$, or \emph{$\mathcal L$-free} for short, if there does not exist any nontrivial solution $(x_1, \ldots, x_k)$ with values chosen from $M$. Here a solution $(x_1, \ldots, x_k)$ is said to be trivial if all $x_i$'s are equal. Given an integer $m$, the question of determining  the largest size of an $\mathcal L$-free
subset $M\subset [m]$
is a fundamental question in additive combinatorics whose history dates back many decades, see e.g. \cite{schur1916kongruenz,erdos1936some,komlos1975linear}.

If $\sum_{i=1}^k a_i=b=0$, the equation $\mathcal L$ is called  {\it invariant}  \cite{ruzsa1993solving}. Invariant equations have drawn a lot of attention so far, since the maximum $\mathcal L$-free subset of $[m]$
must have size $o(m)$ \cite{szemeredi1975sets}. Many fundamental topics in combinatorial number theory can be restated as solution-free problems with different invariant equations, such as  Sidon sets~\cite{ALON1985201,Erdos1986, Furedi2013, o2004complete,erdos1941problem} ($x+y=u+v$) and Roth theorem~\cite{roth1953certain} ($x+y=2z$). Note that the definition of ``trivial'' solutions may be different for different problems.

Besides of considering one linear equation, considering the solution-free set problem of a linear system is also of interest. For example,
a $k$-term arithmetic progression can be described by a system of $k-2$ linear equations with $k$ variables \cite{szemeredi1975sets}.

In this paper, we consider a  \emph{common} solution-free subset $M\subset [m]$ for a collection of related invariant equations. Here, ``common'' means the subset $M$ is solution-free for each equation in the system.  The set of equations is defined by a set $R$ of integers, so we call $M$  an  $R$-solution-free set. Details of the definition are given below.

Given a sequence $U=(u_1, u_2, \ldots, u_k)$ composed by $k$ mutually different integers, we can naturally define an invariant equation $\mathcal L(U)$ with $k$ variables of the form
\begin{equation}\label{eqlu}
  (u_1- u_k)x_1+ (u_2-u_1)x_2+ \ldots+ (u_k- u_{k-1})x_k=0.
\end{equation}
Let $R=\{b_1, b_2, \ldots, b_t\}$ with $0\le b_1<b _2<\cdots <b_t$ be a set of $t$ integers.
A set $M\subset [m]$ is said to be $R$-{\it solution-free} if for any $k\in [3, t]$ and any {$k$-permutation sequence} $U$ of $R$, $M$ is $\mathcal L(U)$-free. Define the rank of $R$ to be $b_t-b_1$, denoted by $r(R)$.

In  \cite[Lemmas 6.1-6.2]{shangguan2016separating}, the authors showed that if $R$ is a $t$-subset in $[0,q-1]$ and $M$ is an $R$-solution-free set in $[0, \lfloor(q-1)/r(R)\rfloor]$, then the $t\times q|M|$ array $\mathcal A$ with columns $(y+b_1m,y+b_2m,\ldots,y+b_tm)^\top$, $y\in \mathbb{{Z}}_q$ and $m\in M$, is a $PHF(t; q|M|, q, t)$. All calculations here are modulo $q$. The proof is done by contradiction in two steps.  If $\mathcal A$ is not a PHF, one can deduce that for some $k\in[t]$, there exist distinct pairs  $(y_1,m_1),\ldots,(y_k,m_k)\in \mathbb{{Z}}_q\times M$ (i.e., $k$ different columns) and distinct indices $j_1,j_2,\ldots,j_k\in [t]$ (i.e., $k$ different rows), such that $y_i+b_{j_{i+1}}m_i=y_{i+1}+b_{j_{i+1}}m_{i+1}$, $i\in [k-1]$ and $y_k+b_{j_1}m_k=y_{1}+b_{j_1}m_{1}$. This is the so called rainbow $k$-cycle defined in  \cite{shangguan2016separating}.  However, these $k$ equations imply that \[(b_{j_2}-b_{j_1})m_1+(b_{j_3}-b_{j_2})m_2+\cdots+(b_{j_k}-b_{j_{k-1}})m_{k-1}+(b_{j_1}-b_{j_k})m_k=0,\] which contradicts that $M$ is an $R$-solution-free set.   These results can be summarized and restated as follows for our purpose.
\begin{lemma}\label{lemma_pro_6.12}
(See \cite{shangguan2016separating}, Lemmas 6.1-6.2.)
Let $R\subset [0, q-1]$ be a set of $t$ elements. If $M\subset [0, \lfloor (q-1)\slash r(R)\rfloor]$ is an $R$-solution-free set, then $p_t(t, q)\ge q|M|$.
\end{lemma}

%

 By Lemma~\ref{lemma_pro_6.12}, we are interested in a $t$-set $R\subset [0,q-1]$ which has a big $R$-solution-free set $M\subset [0,\lfloor(q-1)/r(R)\rfloor]$. For $t=3,4$, the authors chose $R=\{0,1,2\}$ and $R=\{0,2,5,\lceil2^{\sqrt{\log q}}\rceil+5\}$ respectively, and showed that there exists an $R$-solution-free set in $[0, \lfloor (q-1)\slash r(R)\rfloor]$ of size at least $q^{1-o(1)}$ for sufficiently large $q$. Consequently, they were able to prove Theorem~\ref{thm_old_PHF_bounds}. In this paper, we prove the existence of a set $R$ with the same property for all $t\geq 3$, that is Theorem~\ref{thm_sl}. In fact, the set $R$ is constructed explicitly as stated in the following theorem.

\begin{theorem}\label{theorem_main_solution_free}
For a fixed integer $t$ and any reals $a,b$ satisfying $0<b\le a^t<1$, when $m$ is large enough, let $R=\{b_1, \ldots, b_t\}$ be a subset of $[m]$ satisfying
\begin{itemize}
\item[(1)] $\log{b_1}=\Omega(\log^b{m}),$ $\log{b_t}=O(\log^a m);$ and
\item[(2)] for any $i\in [t-1],$ $\log{b_i}=O(\log^a{b_{i+1}})$.
\end{itemize}
Then there exists a constant $c=c(t, a)\in (0, 1)$ such that an $R$-solution-free set $M\subset [m]$ exists with $|M|=\frac m {2^{O(\log^c m)}}$.
\end{theorem}

In Theorem~\ref{theorem_main_solution_free}, the parameter $m$ is supposed to satisfy $m=q^{1-o(1)}$. We use $m$ instead of $q$ since the set $M$ is required  to be contained in $[0,\lfloor(q-1)/r(R)\rfloor]$. So if $m<(q-1)/r(R)$, $M$ satisfies the required range condition. To show that Theorem~\ref{theorem_main_solution_free} indeed implies  Theorem~\ref{thm_sl}, we need the following useful lemma.

\begin{lemma}\label{lemma_little_order}
(See \cite{Ge2021SparseHN}.)
For any fixed $a\in (0, 1)$, $2^{O(log^a m)}= m^{o(1)}.$
\end{lemma}

From Lemma~\ref{lemma_little_order}, the conditions of elements of $R$ in  Theorem~\ref{theorem_main_solution_free} can be roughly explained as follows. The maximum element $b_t=m^{o(1)}$, that is, the rank of $R$ is very small comparing to $m$ or $q$.  For general $i\in [t-1],$ $b_i=b_{i+1}^{o(1)}$, that is $b_i$ is much smaller than $b_{i+1}$. Finally, although $b_1$ is the smallest number, $b_1$ is not too small by $\log{b_1}=\Omega(\log^b{m})$. These conditions are useful for us to prove the existence of a large $R$-solution-free set later.

\begin{proof}[Proof of Theorem~\ref{thm_sl} using Theorem~\ref{theorem_main_solution_free}] Let $m=\lfloor{(q-1)}\slash{2^{\sqrt{\log q}}}\rfloor-1$. Then $m=q^{1-o(1)}<q$ by  Lemma~\ref{lemma_little_order}. Take $R=\{b_1, b_2, \dots, b_t\}$ with $b_t=\lfloor2^{\sqrt{\log{m}}}\rfloor$, $b_{t-1}=\lfloor2^{\sqrt{\log{b_t}}}\rfloor$,  $\ldots,$ $b_2=\lfloor2^{\sqrt{\log{b_3}}}\rfloor$ and $b_1=\lfloor2^{\sqrt{\log{b_2}}}\rfloor$. Then $\log {b_{i+1}}=\Theta(\log^2 b_i)$ for $i\in[t-1]$ and $\log b_t= \Theta(\log^{1\slash 2} m)$. For $i=t-1,\ldots,1$, by substituting the $i$th equation for $b_t$ in the $t$th equation recursively $t$ times, we get  $\log b_1=\Theta(\log^{1\slash 2^t} m)$. Since $t$ is a fixed constant, $1\slash 2^t$ is a constant. By applying Theorem~\ref{theorem_main_solution_free} with $a=\frac 1 2$ and $b=\frac 1 {2^t}$, we get the existence of an $R$-solution-free set $M\subset [m]$ with $|M|=\frac m {2^{O(\log^c m)}}$, where $c$ only depends on $t$. By Lemma~\ref{lemma_little_order}, $|M|=m^{1-o(1)}= q^{1-o(1)}$. Finally, the rank of $R$, $r(R)\leq b_t=\lfloor 2^{\sqrt{\log m}}\rfloor< 2^{\sqrt{\log{q}}}$, so $m<\lfloor(q-1)\slash r(R) \rfloor$.
%
%
\end{proof}

Now we give the proof of Theorem~\ref{thm_all}.

\begin{proof}[Proof of Theorem~\ref{thm_all}]
The upper bound is simply from Theorem~\ref{thm_general_upper}. As for the lower bound, we first have $p_t(t, q)\ge q^{2-o(1)}$ for all $t\geq 3$ by Theorem~\ref{thm_sl} and Lemma~\ref{lemma_pro_6.12}. That is, we obtain a $PHF(t; n, q, t)$ with $n=q^{2-o(1)}$. For general  $SHF(u; n, q, \{w_1, \ldots, w_t\})$ with $u\ge 3$, we consider the refinement of its type, $\{w_1'=1, w_2'=1, \ldots, w_u'=1\}$. Since a $PHF(u; n, q, u)$ is an $SHF(u; n, q, \{w_1, \ldots, w_t\})$ by definition,
we have
$$C(u, q, \{w_1, \ldots, w_t\})\ge C(u, q, \{w_1', \ldots, w_u'\})= p_u(u, q)\ge q^{2-o(1)}.$$
\end{proof}

To prove Theorem~\ref{theorem_main_solution_free}, we need to find a common solution-free set of size $q^{1-o(1)}$ for the family of invariant equations as Eq. (\ref{eqlu}) defined by $R$. Since the size $t$ of $R$ is a fixed finite number, the number of invariant equations considered is finite.  By the following lemma, it is  enough to show that each invariant equation in this family has a solution-free set with size at least $q^{1-o(1)}$.


\begin{lemma}\label{lemma_sum_free_intersection}
(See \cite{Ge2021SparseHN}.)
Let $0<c<1$ be a fixed constant and $l$ be a fixed positive integer. Let $\sum_{i=1}^s a_{ij}x_i=0$, $1\le j\le l$ be $l$ invariant equations (some coefficient $a_{ij}$ may be zero) with unknowns $x_i$.
Suppose that for $1\le j \le l$, there exists $M_j\subset [m]$ with $|M_j|\ge \frac m {2^{O(\log^c{m})}}$ which is solution-free of the $j$th equation $\sum_{i=1}^s a_{ij}x_i=0$. Then, there exists $M\subset [m]$ with $|M|\ge \frac m {2^{O(\log^c{m})}}$ which is solution-free of any of the $l$ equations $\sum_{i=1}^s a_{ij}x_i=0$, $1\le j\le l$.
\end{lemma}

By Lemma~\ref{lemma_sum_free_intersection}, it suffices to show the following key result. 
\begin{theorem}\label{thm_trans_main_thm}
For a fixed integer $t$ and any reals $a,b$ satisfying $0<b\le a^t<1$, when $m$ is large enough,  let $R=\{b_1, \ldots, b_t\}$ be a subset of $[m]$ satisfying
\begin{itemize}
\item[(1)] $\log{b_1}=\Omega(\log^b{m}),$ $\log{b_t}=O(\log^a m);$ and
\item[(2)] for any $i\in [t-1],$ $\log{b_i}=O(\log^a{b_{i+1}})$.
\end{itemize}
Then for any $k$-permutation sequence $U$ of $R$ with {$k\in[3, t]$}, there exists a constant $c=c(\tau, a)\in (0, 1)$ such that an $\mathcal L(U)$-free set $M\subset [m]$ exists with $|M|=\frac m {2^{O(\log^{c} m)}}$. Here $\tau$, which will be defined later, is some number in $[0,t-3]$ depending on $U$.
\end{theorem}

By Theorem~\ref{thm_trans_main_thm} and Lemma~\ref{lemma_sum_free_intersection}, it is immediate to prove Theorem~\ref{theorem_main_solution_free}. The only point needed to explain is the constant $c$. In Lemma~\ref{lemma_sum_free_intersection}, the constant $c$ should be the same for different equations. However in Theorem~\ref{thm_trans_main_thm}, the constant $c$ for different $U$ may be different, which depends on the number $\tau\in[0,t-3]$ defined by $U$.  Since the number of choices of $\tau$ is finite, the set of different values of $c$ is finite.  So we can choose the biggest $c$ among them, which satisfies $|M|=\frac m {2^{O(\log^{c} m)}}$ for all possible $U$.

The remaining of this paper is devoted to proving Theorem~\ref{thm_trans_main_thm}, that is, the existence of a large $\mathcal L(U)$-free set for each possible permutation sequence $U$ from $R$. In \cite{Ge2021SparseHN}, the authors proved the existence of a large solution-free set for any invariant equation with exactly one negative/positive coefficient. See below.

\begin{lemma}\label{corollary_one_side_along}
(See \cite{Ge2021SparseHN}.)
Let $k\ge 2$ be a fixed integer, and let $a\in (0, 1)$ be a fixed real number. Suppose that $k$ integers $a_1, a_2, \ldots, a_k\in [2^{O(\log^a m)}]$. Then there exists $M\subset [m]$ such that $|M|\ge \frac{m}{2^{O(\log^{(1+a)\slash 2} m)}}$  and $M$ is solution-free of the equation $a_1x_1+\cdots+a_k x_k =(a_1+\cdots +a_k) x_{k+1}.$
\end{lemma}

Our main idea to prove Theorem~\ref{thm_trans_main_thm} is to establish a link from $\mathcal L(U)$ to some equation of the form in Lemma~\ref{corollary_one_side_along} in some way, and show that the existence of a large solution-free set could be extended through this link. Details will be given later.

\section{Notations and Outline}\label{sec_notations}

We first introduce several notations that will be used in the proof of Theorem~\ref{thm_trans_main_thm}.
\subsection{Terminating numbers and deletion characters}\label{subsec_cyc_diff_and_ter_num}

For a linear equation $ a_1x_1+ a_2x_2+\cdots+a_kx_k=0$, we call $k$ the length of it, and  say it is of type $(s, r)$ if exactly $s$ coefficients are positive and $r$ coefficients are negative ($s+r=k$).  A  linear equation of type $(s, r)$ can be rearranged to the following form
\begin{equation}\label{eq_4}
a_1x_1+\cdots+a_sx_s=a_1'x_{s+1}+\cdots+a_r'x_{s+r},
\end{equation}
such that all $a_i$ and $a_j'$ are positive. Since reordering the indices in each side does not change the type  and the number of solutions, we can simply denote the equation by a pair of multi-sets  $\mathcal A=(\mathcal A_+; \mathcal A_-)$, where $\mathcal A_+=\{a_1,\ldots,a_s\}$ and $\mathcal A_-=\{a_1',\ldots,a_r'\}$.  We call such a pair of multisets a {\it bipartite array}, or just {\it an array} for short. Notice that we do not view  $(\mathcal A_+; \mathcal A_-)$ and  $(\mathcal A_-; \mathcal A_+)$ as the same array or the same equation, since they may have different types. We usually regard the type $(s, r)$ equation and the corresponding bipartite array as the same object; therefore they share the same type, the same length, sometimes the same notation and other properties.
%

From now on, when we mention a permutation sequence, we always assume that the length is at least three.
Given a permutation sequence $U$ and the linear equation $\mathcal L(U)$, let $\mathcal A(U)=(\mathcal A_+; \mathcal A_-)$ be the corresponding array for $\mathcal L(U)$, and denote $\partial U$  the  multiset union $\mathcal A_+\cup \mathcal A_-$. It is easy to see that $\partial U$ is invariant under the shifting operations on $U$, and so is $\mathcal L(U)$. So we can consider $U$ as a cyclic sequence, and call $U$ increasing if  there exists a shift of $U$ such that the sequence is increasing, and similarly for the decreasing property. Increasing and decreasing permutation sequences are called monotonic.
Note that any $3$-permutation sequence is monotonic. For any monotonic $U$, the linear equation $\mathcal L(U)$ defined by it is of type $(s, 1)$ or $(1, s)$ for some $s\geq 2$, and this is the required type in Lemma~\ref{corollary_one_side_along}.


For any object $S$ consisting of finite numbers, denote $\max S$ (resp. $\min S$) the maximum (resp. minimum) number in $S$.
For a $k$-permutation sequence $U$, denote $U'$ as the subsequence of length $k-1$ by deleting $\max U$ from $U$, and keeping the order of the remaining elements.  For any $j<k$, the resultant subsequence from $U$ by $j$ steps of such deletions is denoted as $U^{(j)}$, i.e., $U^{(1)}:=U'$ and $U^{(j+1)}:=(U^{(j)})'$ for any $j$. Define $U^{(0)}:=U$.
The smallest nonnegative integer $j$ such that $U^{(j)}$ is monotonic is called the {\it terminating number} of $U$, denoted by $\tau(U)$, or simply $\tau$ if there is no confusion. Note that this $\tau$ is exactly the parameter used to define the constant $c$ in Theorem~\ref{thm_trans_main_thm}.

\begin{example}\label{example_terminal_number}
The $8$-permutation sequence $U=(3, 6, 8, 4, 5, 1, 7, 2)$ needs three deletion steps before being monotonic, so $\tau(U)=3$. The subsequences after each deletion are listed below. The final sequence $U^{(3)}$ is increasing.
$$\begin{aligned}
U^{(1)}=&(3, 6, 4, 5, 1, 7, 2);& U^{(2)}=(3, 6, 4, 5, 1, 2);\\
U^{(3)}=&(3, 4, 5, 1, 2).&
\end{aligned}$$
\end{example}

\begin{remark}
Since any $3$-permutation sequence is monotonic, $\tau(U)\le k-3$ for any $k$-permutation sequence $U$. The terminating number $\tau(U)=0$ if and only if $U$ itself is monotonic.
\end{remark}

To record the information in each step of deletion,
we define a {\it deletion character} $\chi(U)$ of $U$ as a binary vector of length $\tau(U)$ as follows. Let $\epsilon(U)$ be an indicator of the terminal status of $U$, i.e., $\epsilon=1$ if $U^{(\tau)}$ is increasing, and $\epsilon=0$ otherwise. For each $i\in [0, \tau]$, denote $U^{(i)}=(u^{(i)}_1, \ldots, u^{(i)}_{k-i})$ and denote $\max(i)$ as the subindex of $\max U^{(i)}$ under $\mathbb Z_{k-i}$. Then for any $j\in [\tau]$, define $\chi(U)_j=1+\epsilon(U) \pmod 2$ if $u^{(j-1)}_{\max(j-1)-1} < u^{(j-1)}_{\max(j-1)+1}$, and $\chi(U)_j=\epsilon(U)$ otherwise.
Under this definiton, in Example~\ref{example_terminal_number}, $U^{(3)}$ is increasing, hence  $\epsilon(U)=1$ and $\chi(U)=(1, 0, 0)$.

\subsection{$(a, b)$-Feasible arrays}


Now we focus only on a set $R=\{b_1, \ldots, b_t\}$ satisfying the conditions in  Theorem~\ref{thm_trans_main_thm}, that is, for a fixed integer $t$ and reals $a,b$ satisfying $0<b\le a^t<1$, when $m$ is large enough,
\begin{itemize}
\item[(1)] $\log{b_1}=\Omega(\log^b{m}),$ $\log{b_t}=O(\log^a m);$ and
\item[(2)] for any $i\in [t-1],$ $\log{b_i}=O(\log^a{b_{i+1}})$.
\end{itemize}
Such a set $R$ is  called {\it $(a, b)$-plastic with respect to $m$}.
{ Note that a set $R$ is $(a, b)$-plastic means that on one hand, the smallest $b_1$ should not be too small with respect to $m$ ensured by the parameter $b$; on the other hand, for each $i\in[t-1]$, $b_i$ should be much smaller than $b_{i+1}$ with some extent bounded by the parameter $a$.} Clearly, such requirements ensure that $b_1< b_2<\cdots <b_t$.

It is worth to mention that the above definition only works for sufficiently large  $m$, and the asymptotic symbols $\Omega, O$ are used when $m$ goes to infinity. For clarity, we sometimes write $R(m)$ to indicate that each element of $R$ is a function of $m$. The similar fashion will happen to the remaining definitions in this paper involving functions of $m$ or asymptotic symbols.

Next, we show that if a set $R$ is $(a, b)$-plastic with respect to large enough $m$, then all invariant equations defined by a permutation sequence of $R$ must satisfy certain conditions, which are called $(a, b)$-feasible below.

\begin{definition}\label{def:fea} For   $0<a,b<1$,  an array $\mathcal A(m)=(\mathcal A_+; \mathcal A_-)$ of an invariant linear equation is called  \emph{$(a, b)$-feasible with respect to $m$} if the following properties hold for sufficiently large $m$.

\begin{itemize}
\item[(1)] The elements in the multiset $\mathcal A_+\cup \mathcal A_-$ are mutually unequal.
\item[(2)] Define $\alpha=\max\mathcal A_+$ and $\alpha'=\max\mathcal A_-$. Then $|\alpha'- \alpha|=o(\alpha)$ and $\log {\alpha}=O(\log^a m)$.
\item[(3)] Define the set $Z(\mathcal A)=\mathcal A_+\cup \mathcal A_-\cup\{|\alpha-\alpha'|\}\setminus\{\alpha, \alpha'\}$. Then for any $w\in Z(\mathcal A)$, $\Omega(\log^b m)=\log w= O(\log^a \alpha)$. 
\end{itemize}
\end{definition}

%
%

By Definition~\ref{def:fea}, if a bipartite array $\mathcal A$ is $(a, b)$-feasible with respect to $m$, it is also $(a, b')$-feasible with respect to $m$ for any $b'$ satisfying $0< b'\le b$. The definition of $Z(\mathcal A)$ looks weird. However, it is designed to be a set approximate to the main part of an array which is highly related to $\mathcal A$ and is useful in our recurrence proof of Theorem~\ref{thm_trans_main_thm}. See the ancestors in Definition~\ref{def_ancestor}.


Before proving that all invariant equations defined by $R$ are $(a, b)$-feasible, we present the
 following lemma, which is useful to estimate the order of any linear combinations of elements from $R$, $U$ or $\mathcal A$.
\begin{lemma}\label{lemma_diff_easy}
Let $R=\{b_1, \ldots, b_t\}$ be an $(a, b)$-plastic set with respect to  $m$, and let $d$ be some fixed positive integer. Suppose that $w=\sum_{i=1}^t\omega_i b_i$ with coefficients $\omega_1, \ldots, \omega_t\in\mathbb Z$ and $0<\sum_{i=1}^t|\omega_i|\le d$. Then
 $|w|=\Theta(b_p)$, where $p\in [t]$ is the maximum index satisfying $\omega_p\ne 0$. Consequently, $w$ is nonzero. 
\end{lemma}
\begin{proof}
For any given positive integer $d$ and $\omega_1, \ldots, \omega_t\in\mathbb Z$, define $p=\max\{j\in [t]: \omega_j\ne 0\}$. Thus, $w=\sum_{i=1}^p\omega_i b_i$. From Lemma~\ref{lemma_little_order}, for any $j<p$, $b_j=o(b_p)$. In the mean time, for each $i\in [p]$, $|\omega_i|\le d$, which is a positive constant. Thus, $|\sum_{i=1}^{p-1}\omega_ib_i|\le d\sum_{i=1}^{p-1}b_i\le (p-1)d\cdot b_{p-1}= o(b_p)$, and $w=\omega_p b_p\pm o(b_p)$ with some $\omega_p\ne 0$. So $|w|=\Theta(b_p)$.
\end{proof}

\begin{lemma}\label{lemma_star_U_is_feasile} Let $R$ be an $(a, b)$-plastic set with respect to $m$ of size $t$ for some $a, b$ satisfying $0<b\le a^t<1$. Then for any $k\in [3, t]$ and for any $k$-permutation sequence $U$ of $R$, the array $\mathcal A(U)$  is $(a, b)$-feasible with respect to $m$.
\end{lemma}

\begin{proof} Let $R=\{b_1< \cdots< b_t\}$. By the definition of $(a, b)$-plasticity and  Lemma~\ref{lemma_little_order}, as functions of $m$, $\max\{b_j, j<i\}=o(b_i)$ for any $i\in [t]$. Given two different elements $v$ and $w$ in $\mathcal A(U)$,  they can be expressed by $v=b_i-b_j$ and $w=b_{i'}-b_{j'}$ for two different pairs of different indices $(i, j)$ and $(i', j')$ in $[t]$. So $v-w$ is a linear combination of $\{b_1, \ldots, b_t\}$ satisfying Lemma~\ref{lemma_diff_easy} with $d=4$. Hence  $v-w\neq 0$ and elements in $\mathcal A(U)$ are mutually unequal.

Denote $U=(u_1, \ldots, u_k)$. Without loss of generality, we set $u_1$ be the largest element in $U$ and hence $u_i=o(u_1)$ for any $i\in [2, k]$. By the definition of $\mathcal A(U)$, it is easy to check that the largest two elements in $\mathcal A(U)_+$ and $\mathcal A(U)_-$ are  $\alpha= u_1- u_k=\Theta (u_1)$ and $\alpha'= u_1- u_2=\Theta (u_1)$, respectively.
 So $\log \alpha= \Theta(\log u_1)= O(\log b_t)= O(\log^a m)$ and $|\alpha-\alpha'|= |u_k -u_2|= o(u_1)= o(\alpha)$.

For elements in $Z(\mathcal A(U))$, $|\alpha- \alpha'|=|u_k -u_2|= \Theta({\max\{u_k, u_2\}})$. So $\log |\alpha- \alpha'|=O(\log^a u_1)= O(\log^a \alpha)$ and $\log |\alpha- \alpha'|\ge \Omega(\log b_1)\ge \Omega(\log^b m).$ All other elements in $Z(\mathcal A(U))$ have the form $w=u_i - u_j$ with $i\neq 1\neq j$. So $w=\Theta({\max\{u_i, u_j\}})$, which has the same bounds as $|\alpha- \alpha'|$.
\end{proof}



\subsection{An outline of the proof for Theorem~\ref{thm_trans_main_thm}}

Suppose that $m'=m'(m)$ is a positive integer-valued function which goes to infinity with $m$. Given $a, b\in (0, 1)$ and an integer $k\ge 3$, we say an $(a, b)$-feasible array $\mathcal A=\mathcal A(m)$ of length $k$ with respect to $m$ is  \emph{$c$-good under location $m'$} for some $c=c(k, a, b)\in (0, 1)$, if there exists a subset $M\subset [m']$ such that $M$ is solution-free of $\mathcal A$, and $|M|\ge\frac{m'}{2^{O(\log^c m')}}$.
Under this definition, we can rewrite Theorem~\ref{thm_trans_main_thm} in the following way.
\begin{thmbis}{thm_trans_main_thm} For a fixed integer $t$ and any reals $a,b$ satisfying $0<b\le a^t<1$, when $m$ is large enough, let $R(m)$ be an $(a, b)$-plastic set with respect to $m$ of size $t$.
 For any $k$-permutation sequence $U$ of $R$ with $3\le k\le t$, the  array $\mathcal A(U)$ is $c$-good under location $m$ for some $c=c(\tau(U), a)\in (0, 1)$.
\end{thmbis}

By Lemma~\ref{lemma_star_U_is_feasile}, $\mathcal A(U)$ is an $(a, b)$-feasible array. The outline of the proof for Theorem~\ref{thm_trans_main_thm} is as follows.

First, we point out a {\it goodness inheriting relationship} among $(a, b)$-feasible arrays with respect to a common $m$. That is to say, for any non-monotonic $(a, b)$-feasible array $\mathcal A$, we can find two kinds of other arrays $\mathcal A_1$ and $\mathcal A_2$, which we call ancestors of $\mathcal A$, such that if any ancestor $\mathcal A_i$ is $c$-good under a certain location, then $\mathcal A$ is $c$-good under a wide range of locations. If this happens, we say the goodness of $\mathcal A$ inherits from its ancestors. See the link lemma (Lemma~\ref{lemma_linkage}).

Then we show that for any $k$-permutation sequence $U$ of some $(a, b)$-plastic $R$ with respect to some large enough $m$, we can always expand a long string of $(a, b')$-feasible arrays for some $b'\le b$ with respect to $m$ from $\mathcal A(U)$, denoted by $\mathcal A_{[0]}= \mathcal A(U), \mathcal A_{[1]}, \ldots, \mathcal A_{[\tau]}$, such that each succeeding array is always an ancestor of the preceding one (see Theorem~\ref{thm_induc_link_lemma}), and the last array is always monotonic (see Lemma~\ref{string_lemma_3}).

From the goodness of monotonic arrays by Lemma~\ref{corollary_one_side_along}, we can inherit the goodness property along the string of ancestors, finally resulting in the goodness of $\mathcal A(U)$ under the location $m$ (see Proof of Theorem~\ref{thm_trans_main_thm}). The next two sections are the detailed proving process following this outline.




\section{The Link Lemma}\label{sec_link_lemma}

For a given non-monotonic $(a, b)$-feasible array $\mathcal A$ with respect to $m$, we find that the goodness of some other array with respect to the same $m$ can lead to the goodness of $\mathcal A$. Such an array is called an {\it ancestor} of $\mathcal A$. For a given $\mathcal A$, it may have many ancestors in different types. Next we give a detailed definition for two special types of ancestors for each non-monotonic feasible $\mathcal A$.

\begin{definition}\label{def_ancestor}
Given any large enough $m$ and a non-monotonic $(a, b)$-feasible array $\mathcal A(m)=(\mathcal A_+; \mathcal A_-)$ for some fixed $a, b\in (0, 1)$, define $\alpha=\max\mathcal A_+$ and $\alpha'=\max\mathcal A_-$ as in Definition~\ref{def:fea}.
Define $\delta(m)=\min(\mathcal A_+\cup \mathcal A_-\cup\{|v-w|: v\ne w\in \mathcal A\})$, which is determined by $\mathcal A(m)$.
A positive integer $\theta(m)$  is called \emph{reproducible} for $\mathcal A(m)$ if $\theta$ goes to infinity with $m$ and  $\theta(m)=O(2^{\log^a \delta(m)})$. Suppose there exists a reproducible $\theta$ for $\mathcal A(m)$. \\
If $\alpha'>\alpha$, define
\begin{equation*}
\begin{array}{ll}
\mathcal A_1=& (\mathcal A_+\backslash\{\alpha\}; \mathcal A_-\backslash\{\alpha'\}\cup \{\alpha'-\alpha-\theta, \theta\});\\
\mathcal A_2=& (\mathcal A_+\backslash\{\alpha\}\cup\{\theta\}; \mathcal A_-\backslash\{\alpha'\}\cup \{\alpha'-\alpha+\theta\}).
\end{array}
\end{equation*}
Otherwise, since numbers in $\mathcal A$ are mutually unequal, $\alpha'<\alpha$. In this case, define
\begin{equation*}
\begin{array}{ll}
\mathcal A_1=& (\mathcal A_+\backslash\{\alpha\}\cup\{\alpha-\alpha'-\theta, \theta\}; \mathcal A_-\backslash\{\alpha'\});\\
\mathcal A_2=& (\mathcal A_+\backslash\{\alpha\}\cup\{\alpha-\alpha'+\theta\}; \mathcal A_-\backslash\{\alpha'\}\cup \{\theta\}).
\end{array}
\end{equation*}
The unions and exclusions above are multiset operations. By the choice of $\theta$, all numbers in $\mathcal A_1$ and $\mathcal A_2$ for both cases are positive, and thus $\mathcal A_1$ and $\mathcal A_2$ are bipartite arrays. We call $\mathcal A_1$ and $\mathcal A_2$ the first and the second type of \emph{ancestors of $\mathcal A$ by $\theta$}, respectively.
\end{definition}

\begin{example}
  To show that the two types of ancestors of $\mathcal A$ by $\theta$ exist, we only need to show that the chosen $\theta$ is reproducible for $\mathcal A$.
For example, suppose $\mathcal A=\mathcal A(U)$ for some permutation sequence $U$ of an $(a,b)$-plastic set. Let $S=\mathcal A(U)_+\cup \mathcal A(U)_-\cup\{|v-w|: v\ne w\in \mathcal A(U)\}$. Then any element in $S$ can be expressed as a linear combination of at most four members in $U$. By Lemma~\ref{lemma_diff_easy} with $d=4$, any element $w$ in $S$ is positive with order $w=\Omega(\min U)$. As a result, $\delta=\min S=\Omega(\min U)= 2^{\Omega(\log^b m)}$, which goes to infinity as $m$ goes to infinity. So $\theta=\lfloor2^{\log^a (\min U)}\rfloor=O(2^{\log^a \delta})$ is reproducible for $\mathcal A(U)$.
\end{example}

\begin{remark}\label{remark_of_link_lemma} 
We have the following observations about Definition~\ref{def_ancestor}.
\begin{itemize}
\item[(1)] If $\mathcal A$ is of $(s, r)$ type, then $\mathcal A_1$ and $\mathcal A_2$ are both invariant with types from $(s-1, r+1)$, $(s, r)$ and $(s+1, r-1)$.
\item[(2)] The element $\theta$ is a common element in $\mathcal A_1$ and $\mathcal A_2$. As $\delta(m)$ goes to infinity, from Lemma~\ref{lemma_little_order}, $\theta=o(\delta)$ and hence $\theta$ is the smallest element in both $\mathcal A_1$ and $\mathcal A_2$.
\end{itemize}
\end{remark}

Now we give the link lemma, where notations have  the same meaning as in Definition~\ref{def_ancestor}.

\begin{lemma}\label{lemma_linkage} (Link lemma.) Let $0< a, b<1$ be fixed, and let $\mathcal A(m)$ be any non-monotonic $(a, b)$-feasible array with respect to some large enough $m$. Suppose $\mathcal A_1$ and $\mathcal A_2$ are the first and the second type of ancestors of $\mathcal A$ by some integer $\theta$.
If either $\mathcal A_1$ or $\mathcal A_2$ is $c$-good under location $\lfloor \min\{\alpha, \alpha'\}\slash \sigma\rfloor$, where $\sigma$ is the element sum in one side of the corresponding ancestor, then $\mathcal A$ is $c'$-good for some $c'= c'(c, a)$ under any location $m'$ satisfying $\log \alpha=O(\log^a m')$.
\end{lemma}
\begin{proof}
Suppose $\mathcal A(m)=(\{a_1, \ldots, a_s\}; \{a_1', \ldots, a_r'\})$  for some $s, r\ge 2$, and two maximum elements $\alpha=a_s$, $\alpha'=a_r'$ for both sides.  We only prove the case when $\alpha'>\alpha$. To get the other half result, one only needs to interchange the two parts of $\mathcal A$ and change $\alpha$ in the following analysis by $\alpha'$.

Suppose $\mathcal A_1$ is $c$-good for some fixed $0<c<1$. Then $\sigma=\sum_{i=1}^{s-1}a_i.$ From the $(a, b)$-feasibility of $\mathcal A$ and Lemma~\ref{lemma_little_order}, $\sigma=o(\alpha)$ and $\log \sigma=O(\log^a \frac{\alpha}{\sigma})$.

The location choice $\lfloor \min\{\alpha, \alpha'\}\slash \sigma\rfloor$ goes to $\lfloor\alpha\slash\sigma\rfloor$ in this case, which goes to infinity with $m$ from $\sigma=o(\alpha)$ and $\alpha=\exp(\Omega(\log^b m))$. For the $c$-goodness of $\mathcal A_1$ under location $\lfloor\alpha\slash\sigma\rfloor$, there exists a subset of $[\lfloor\alpha\slash\sigma\rfloor]$, denoted by $B$, such that $B$ is $\mathcal A_1$-free  and $|B|\ge \frac{\alpha}\sigma\slash 2^{O(\log^c (\frac{\alpha}\sigma))}$. So
\begin{equation}\label{ineq_3}
\begin{aligned}
\log |B|\ge & \log (\alpha)-\log \sigma- O(\log^c (\frac {\alpha}\sigma))\\
\ge & \log (\alpha)-O(\log^a \alpha)- O(\log^c \alpha)\\
= & \log(\alpha)- O(\log^{\max\{a, c\}}\alpha).
\end{aligned}
\end{equation}

We expand the numbers of $[m']$ in  base $(\alpha+\theta)$. Denote $\ell= \lfloor\frac{\log m'}{\log (\alpha+\theta)}\rfloor$, which is the largest number of digits that all expansions of numbers in $[0, m'-1]$ can cover.  
Let $M$ consist of those numbers in $ [(\alpha+\theta)^\ell-1]$ whose development in base $\alpha+\theta$ contains only the digits in $B-1=\{b-1:b\in B\}$. That is, if we map a number $p=\sum_{i=0}^{\ell-1}\lambda_i(\alpha+\theta)^i\in [(\alpha+\theta)^\ell-1]$  to a vector $\bm\lambda(p)=(\lambda_0, \lambda_1, \ldots, \lambda_{\ell-1})\in\mathbb Z_{\alpha+\theta}^\ell$, $M$ consists of all vectors of length $\ell$ over $B-1$.


Next we check that $M$ is indeed solution-free of $\mathcal A$. If there exist $s+r$ integers $m_1, m_2, \ldots, m_s,$ $ m_{s+1}, \ldots, m_{s+r}$ in $M$ forming a solution to $\mathcal A$, we need to show that all of them are equal. For any $i\in [s+r]$, denote the development of $m_i$ as $\bm\lambda(m_i)=(\lambda_{0, i}, \lambda_{1, i}, \ldots, \lambda_{\ell-1, i})$. Denote $h$ as the smallest number in $[0,\ell-1]$ such that $\lambda_{h, 1}, \ldots, \lambda_{h, (s+r)}$ are not all the same. Then equation $\mathcal A$  modulo ${(\alpha+\theta)^{h+1}}$ becomes
$$a_1\lambda_{h, 1}+a_2\lambda_{h, 2}+\cdots +a_s\lambda_{h, s}\equiv a_1'\lambda_{h, (s+1)}+\cdots+ a_r'\lambda_{h, (s+r)}\pmod{(\alpha+\theta)}.$$
Notice that $\alpha=a_s$, so
$$\sum_{i=1}^{s-1}a_i\lambda_{h, i} \equiv \sum_{j=1}^{r-1} a_j'\lambda_{h, (s+j)}+ (a_r'-a_s-\theta)\lambda_{h, (s+r)}+ \theta\lambda_{h, s}\pmod{(\alpha+\theta)}.$$
Since each $\lambda_{i, j}\in B-1\subset  [0, \lfloor{\alpha}\slash \sigma\rfloor-1],$  both sides of the above equation are  upper bounded by $(\lfloor\frac{\alpha}\sigma\rfloor-1)\sigma< \alpha+ \theta$. So the congruence modulo symbol actually implies the equivalence,
$$\sum_{i=1}^{s-1}a_i\lambda_{h, i}= \sum_{j=1}^{r-1} a_j'\lambda_{h, (s+j)}+ (a_r'-a_s-\theta)\lambda_{h, (s+r)}+ \theta\lambda_{h, s}.$$
This means $\lambda_{h, 1}, \ldots, \lambda_{h, (s+r)}$, as elements in $B-1$, form a solution to $\mathcal A_1$. Since $B$ is $\mathcal A_1$-free  and $\mathcal A_1$ is invariant, $B-1$ is also $\mathcal A_1$-free. Hence $\{\lambda_{h, i}\}_{ i\in [s+r]}$ is a trivial solution,
contradicting to that $\lambda_{h, 1}, \ldots, \lambda_{h, (s+r)}$ are not all the same. This means such an $h$ does not exist and $m_1, m_2, \ldots, m_{s+r}$ are all equal.

The only thing left is to give a lower bound on the size of $M$. By the definition,
$$|M|=|B|^\ell\ge |B|^{\frac{\log m'}{\log (\alpha+\theta)}-1}=\Omega({m'}^{\frac{\log |B|}{\log (\alpha+\theta)}-\frac{\log |B|}{\log m'}}).$$
Since $\log \alpha = O(\log^a m')$, $\frac{\log |B|}{\log m'}\le \frac{\log \alpha}{\log m'} =O(\log^{\frac{a-1}a} \alpha)$. If we denote $c_1=\max\{\max(a, c)-1, \frac{a-1}{a}\}<0$, $c_2= ac_1$ and $c'=c_2+1\in (0, 1)$,  then
$$
\begin{aligned}
|M|=& \Omega({m'}^{\frac{\log |B|}{\log (\alpha+\theta)}-\frac{\log |B|}{\log m'}})\overset{Eq. (\ref{ineq_3})}{=} \Omega({m'}^{1-O(\log^{c_1}\alpha)})\\
=&\frac {m'} {{m'}^{O(\log^{c_1}\alpha)}}\ge \frac {m'} {{m'}^{O(\log^{c_2} {m'})}}=\frac {m'} {2^{O(\log^{c'} m')}}.
\end{aligned}$$
So $\mathcal A$ is $c'$-good under location $m'$.

As for the case of $\mathcal A_2$, we use the base $(\alpha- \theta)$ rather than $(\alpha+ \theta)$. Here $\sigma= \sum_{i=1}^{s-1}a_i+\theta$. The solution-free set $B$ of $\mathcal A_2$, {$\ell=\lfloor\frac{\log m'}{\log(\alpha-\theta)}\rfloor$}, and the candidate solution-free set $M=\big\{ \sum_{i=0}^{\ell-1}\lambda_i(\alpha-\theta)^i: \lambda_i\in B-1  \big\}$ are similarly defined. The size of $M$ can be bounded as in the case of $\mathcal A_1$ by the same analysis. Here we only give the proof of checking that $M$ is $\mathcal A$-free as follows, while other process is left to the reader.

If there exist $s+r$ integers $m_1, m_2, \ldots, m_{s+r}$ in $M$ forming a solution to $\mathcal A$, denote $\bm\lambda(m_i)=(\lambda_{0, i}, \lambda_{1, i}, \ldots, \lambda_{\ell-1, i})$ for any $i\in [s+r]$, and denote $h$ as the smallest number in $[0,\ell-1]$ such that $\lambda_{h, 1}, \ldots, \lambda_{h, (s+r)}$ are not all the same.  Then equation $\mathcal A$  modulo {$(\alpha-\theta)^{h+1}$} becomes 
$$a_1\lambda_{h, 1}+a_2\lambda_{h, 2}+\cdots +a_s\lambda_{h, s}\equiv a_1'\lambda_{h, (s+1)}+\cdots+ a_r'\lambda_{h, (s+r)}\pmod{(\alpha- \theta)}.$$
This leads to
$$\sum_{i=1}^{s-1}a_i\lambda_{h, i} + \theta\lambda_{h, s}\equiv \sum_{j=1}^{r-1} a_j'\lambda_{h, (s+j)}+ (a_r'-a_s+\theta)\lambda_{h, (s+r)}\pmod{(\alpha- \theta)},$$by noticing that $\alpha=a_s$.
Since both sides are all upper bounded by $\alpha- \theta$, the congruence modulo symbol can be replaced by the equivalence symbol,
$$\sum_{i=1}^{s-1}a_i\lambda_{h, i} + \theta\lambda_{h, s}= \sum_{j=1}^{r-1} a_j'\lambda_{h, (s+j)}+ (a_r'-a_s+\theta)\lambda_{h, (s+r)}.$$
Then as elements in $B-1$, $\lambda_{h, 1}, \ldots, \lambda_{h, (s+r)}$ form a solution to $\mathcal A_2$, which must be trivial. Then by the same analysis as in $\mathcal A_1$, such an $h$ does not exist and $m_1, m_2, \ldots, m_{s+r}$ are equal.
\end{proof}
\begin{remark}\label{remark_parameter_c}
By the proving process of the link lemma, the exact expression of $c'$  is $c'=a\max\{\max\{a, c\}-1, \frac{a-1}a\}+1= 1-a+ a\max\{a, c\}$. Here we do not take efforts to find the best $c'$ satisfying the link lemma.
\end{remark}

Given three positive integers determined by $m$ for some large enough $m$: $h=h(m)$, $f=f(m)$ and $g=g(m)$, we say $g$  {\it approximates} $f$ with the scale $h$, denoted by $g\overset h\approx f$, if $|g-f|=O(h)$. It is easy to see that the approximation relationship with the same scale is reversible and transitive through finitely many steps. This definition is not trivial only when the scale $h$ is much smaller than the evaluated $f$ and $g$. Given two arrays $\mathcal A$ and $\mathcal A'$ whose elements are determined by $m$, for example, they are all $(a, b)$-feasible with respect to $m$, we say $\mathcal A$ approximates $\mathcal A'$ with the scale $h(m)$ if they have the same type and the corresponding coefficients approximate each other with the common scale $h$. The approximation relationships of other structures, such as sets and multisets, can be similarly defined.


The following lemma shows that if $\mathcal A$ is an array considered in Theorem~\ref{thm_trans_main_thm}, then its ancestors have similar structures to the first step deletion of it.

\begin{lemma}\label{le-appro}
Suppose $R(m)$ is $(a, b)$-plastic with respect to some large enough $m$, and $U$ is a $k$-permutation sequence of $R$.
Under any given reproducible $\theta$, the two ancestors $\mathcal A_1$ and $\mathcal A_2$ of $\mathcal A(U)$ by $\theta$ satisfy \begin{equation}\label{eq-scal}
               \mathcal A_1\backslash\{\theta\}\overset\theta\approx \mathcal A (U^{(1)})\overset\theta\approx \mathcal A_2 \backslash\{\theta\}.
             \end{equation}
\end{lemma}
\begin{proof}
Denote $U=(u_1, \ldots, u_k)$ and  $\mathcal A=(\{a_1, \ldots, a_s\}; \{a_1', \ldots, a_r'\})$ with $k=r+s$. For the existence of ancestors, $r, s\ge 2$. From $(a, b)$-feasibility of $\mathcal A$, elements in $\mathcal A$ are mutually unequal. Suppose the two maximum elements in both sides of $\mathcal A$ are $\alpha=a_s$ and $\alpha' = a_r'$, respectively, satisfying $\alpha'<\alpha$. Let $u_1$ be the maximum element in $U$. Then $a_s= u_{1}-u_{k}$, $a_r'= u_{1}-u_{2}$,  $u_{2}- u_{k}= a_s- a_r'\in\mathcal A(U^{(1)})_+$ and $\mathcal A(U^{(1)})=(\{a_1, \ldots, a_{s-1}, a_s-a_r'\}; \{a_1', \ldots, a_{r-1}'\})$. By definition, $\mathcal A_1=(\mathcal A_+\backslash\{\alpha\}\cup\{\alpha-\alpha'-\theta, \theta\}; \mathcal A_-\backslash\{\alpha'\})= (\{a_1, \ldots, a_{s-1}, a_s-a_r'-\theta, \theta\}; \{a_1', \ldots, a_{r-1}'\})$,  so $$\mathcal A_1\backslash\{\theta\} = (\{a_1, \ldots, a_{s-1}, a_s-a_r'-\theta\}; \{a_1', \ldots, a_{r-1}'\})\overset\theta\approx \mathcal A (U^{(1)}).$$ Similarly, $\mathcal A_2=(\mathcal A_+\backslash\{\alpha\}\cup\{\alpha-\alpha'+\theta\}; \mathcal A_-\backslash\{\alpha'\}\cup\{\theta\})= (\{a_1, \ldots, a_{s-1}, a_s-a_r'+\theta,\}; \{a_1', \ldots, a_{r-1}', \theta\})$ and hence $$\mathcal A_2\backslash\{\theta\} = (\{a_1, \ldots, a_{s-1}, a_s-a_r'+\theta\}; \{a_1', \ldots, a_{r-1}'\})\overset\theta\approx \mathcal A (U^{(1)}).$$ The case when $\alpha'>\alpha$ can be similarly checked.\end{proof}

In general, the phenomenon in Lemma~\ref{le-appro} can go inductively and we have the following theorem.

\begin{theorem}\label{thm_induc_link_lemma} Let $R(m)$ be an $(a, b)$-plastic set of size $t$ with respect to some large enough $m$, and let $U$ be some non-monotonic $k$-permutation sequence of $R$ with $k>3$. Then
we can always generate a string 
 $$\mathcal A(U)=\mathcal A_{[0]}, \mathcal A_{[1]}, \ldots, \mathcal A_{[\tau(U)]}$$
by a sequence of choices of theta:
$\theta_1:=\lfloor2^{\log^a (\min U)}\rfloor$, and $\theta_{j+1}:=\lfloor2^{\log^a \theta_j}\rfloor$ for  $j\in[\tau-1]$,
such that for each $i\in [\tau]$:
\begin{itemize}
\item[(1)] $\mathcal A_{[i]}$ can be any type of ancestors of $\mathcal A_{[i-1]}$ by $\theta_i$;
\item[(2)] $\theta_1, \ldots, \theta_i$ are all contained in array $\mathcal A_{[i]}$ as coefficients (may be in different parts);
\item[(3)] $\mathcal A_{[i]}\backslash\{\theta_1, \ldots, \theta_i\}\overset{\theta_1}\approx \mathcal A (U^{(i)})$;
\item[(4)] $\mathcal A_{[i]}$ is $(a, ba^{i})$-feasible with respect to $m$.
\end{itemize}
\end{theorem}
Theorem~\ref{thm_induc_link_lemma} is our important structural theorem for a string of bipartite arrays generated by choosing ancestors iteratively. It reveals that for each $i$, the array $\mathcal A_{[i]}$ after deleting all small coefficients $\theta_1, \ldots, \theta_i$ is almost the invariant array $\mathcal A(U^{(i)})$, which has good feasibility by Lemma~\ref{lemma_star_U_is_feasile}. Hence we can ensure that all $\mathcal A_{[i]}$ are with good feasibility, and one can apply link lemma iteratively.

\subsection{Proof of Theorem~\ref{thm_induc_link_lemma}}
In this subsection, the notations $m$, $R$, $U$, $\mathcal A_{[i]}$ and $\theta_i$ are all from Theorem~\ref{thm_induc_link_lemma}. Since $U$ is non-monotonic, $\tau(U)>0$. We prove Theorem~\ref{thm_induc_link_lemma} by induction on $i\in [\tau]$.
 We begin with some simple but useful results.
\begin{lemma}\label{lemma_small_property}
Let $w=\sum_{i=1}^k\omega_iu_i$ be a linear combination of elements in $U$ with coefficients $\omega_1, \ldots, \omega_k\in\mathbb Z$. As long as $\sum_{i=1}^k|\omega_i|$ is nonzero and bounded by a constant, we have $\theta_1=o(|w|)$.

Moreover, all elements in $\{\theta_1, \ldots, \theta_\tau\}$ go to infinity with $m$, and the order between them is fixed: $\theta_1>\cdots>\theta_\tau$.
\end{lemma}
\begin{proof}
From $(a, b)$-plasticity of $R$ with respect to $m$, $\min U\ge \min R= 2^{\Omega(\log^b m)}$ goes to infinity as $m$ goes to infinity. For Lemma~\ref{lemma_little_order}, $\theta_1=o(\min U)$. From Lemma~\ref{lemma_diff_easy}, we get $|w|=\Omega(\min U)$. Hence $\theta_1= o(|w|)$. The second part is trivial from Lemma~\ref{lemma_little_order}. 
\end{proof}

\begin{corollary}\label{coro_order_fix}
If $v$ and $w$ are elements from array $\mathcal A(U^{(\ell)})$ or set $Z(\mathcal A(U^{(\ell)}))$ for some $\ell\in [0, \tau]$, and $\hat v$, $\hat w$ are integers approximate to $v$ and $w$ with a common scale $\theta_1$, respectively.
Then $v> w$ implies that $\hat{v}>\hat{w}$.
\end{corollary}
\begin{proof} 
By Lemma~\ref{lemma_small_property}, $\theta_1=o(|v-w|)$. Notice that $\hat{v}-\hat{w}=v-w+O(\theta_1)$, so the result is verified for that $m$ is large enough.
\end{proof}

For the base case $i=1$,  properties (1) and (3) have been proved in Lemma~\ref{le-appro}. Property (2) is trivial. Now we prove property (4), that is the following claim.

\begin{claim}\label{claim_a1}
$\mathcal A_{[1]}$  is $(a, ba)$-feasible with respect to $m$.
\end{claim}
\begin{proof}
 By Lemma~\ref{lemma_star_U_is_feasile}, $\mathcal A (U^{(1)})$ is $(a,b)$-feasible with respect to $m$, and hence
  \begin{itemize}
    \item[(i)] elements in $\mathcal A (U^{(1)})$ are mutually unequal;
    \item[(ii)] the largest elements in both sides $\alpha\in \mathcal A(U^{(1)})_+$ and $\alpha'\in \mathcal A(U^{(1)})_-$ satisfy $|\alpha'- \alpha|= o(\alpha)$ and $\log \alpha= O(\log^a m)$;
    \item[(iii)] if we define set $Z_1=Z(\mathcal A(U^{(1)}))$, i.e., $Z_1=\mathcal A(U^{(1)})\cup\{|\alpha- \alpha'|\}\backslash\{\alpha, \alpha'\}$, for any $w\in Z_1$, $\Omega(\log^b m)=\log w=O(\log^a \alpha)$.
  \end{itemize}

   By Lemma~\ref{le-appro}, $\mathcal A_{[1]}\backslash\{\theta_1\}\overset{\theta_1}\approx \mathcal A (U^{(1)})$. Then by Corollary~\ref{coro_order_fix} and (i), elements in $\mathcal A_{[1]}\backslash\{\theta_1\}$ are mutually unequal and possess the same ordering as in $\mathcal A(U^{(1)})$.  By Lemma~\ref{lemma_small_property}, $\theta_1$ is smaller than any one in $\mathcal A_{[1]}\backslash\{\theta_1\}$, so elements in $\mathcal A_{[1]}$ are mutually unequal.

   Denote the corresponding coefficients of $\alpha$ and $\alpha'$ under the approximation relationship in $\mathcal A_{[1]}$ as $\hat{\alpha}$ and $\hat{\alpha}'$, respectively. From Corollary~\ref{coro_order_fix}, they are also the maximum elements in corresonding parts. That is, $\hat{\alpha}=\max \mathcal A_{[1]+}$, $\hat{\alpha}'=\max \mathcal A_{[1]-}$; and $\hat{\alpha}= \alpha\pm O(\theta_1)$, $\hat{\alpha}'= \alpha'\pm O(\theta_1)$. So $|\hat{\alpha}'-\hat{\alpha}|= |\alpha' -\alpha| \pm O(\theta_1)=o(\alpha)=o(\hat{\alpha})$, and $\log\hat{\alpha}=\Theta(\log \alpha)=O(\log^a m)$.

Denote sets $Z= Z(\mathcal A_{[1]})= \mathcal A_{[1]}\cup\{|\hat{\alpha}-\hat{\alpha}'|\}\backslash \{\hat{\alpha}, \hat{\alpha}'\}$ and $Z_2= Z\setminus \{\theta_1\}= (\mathcal A_{[1]}\backslash\{\theta_1\})\cup\{|\hat{\alpha}-\hat{\alpha}'|\}\backslash \{\hat{\alpha},\hat{\alpha}'\}$.  It is easy to see that for any element $\hat w\in Z_2$, there always exists a corresponding element $w\in Z_1$ such that $\hat w\approx w$ with scale $\theta_1$. From Lemma~\ref{lemma_small_property} $\hat w= w\pm o(w)$, and from (iii) $\Omega(\log^b m)=\log w=O(\log^a \alpha)=O(\log^a \hat\alpha)$. So $\Omega(\log^{b} m)=\log \hat w=O(\log^a \hat\alpha)$. Finally, for the definition of $\theta_1$, it is trivial to check $\Omega(\log^{ba} m)=\log \theta_1 =O(\log^a \hat\alpha)$. Hence, $\mathcal A_{[1]}$ is $(a, ba)$-feasible.
\end{proof}


For  $i\ge 2$, all properties can be proved similarly. Suppose we have proved Theorem~\ref{thm_induc_link_lemma} with properties (1)-(4) for any $i\in [\ell]$ for some $1\leq \ell <\tau$. So we have
\begin{itemize}
  \item[(a)]$\theta_1, \ldots, \theta_{\ell}$ are all contained in array $\mathcal A_{[\ell]}$ as coefficients;
  \item[(b)]$\mathcal A_{[\ell]}\backslash\{\theta_1, \ldots, \theta_\ell\}\overset{\theta_1}\approx \mathcal A (U^{(\ell)})$;
  \item[(c)]$\mathcal A_{[\ell]}$ is $(a, ba^{\ell})$-feasible with respect to $m$.
\end{itemize}

Let us consider the case $i={\ell+1}$, and prove that properties (1)-(4) hold for ${\ell+1}$.

{
As an extension of $U$, $\{u_1, \ldots, u_k, \theta_1, \ldots, \theta_\ell\}$ is a new $(a, ba^\ell)$-plastic set with respect to $m$ with size $k+\ell<2k$. By applying Lemma~\ref{lemma_diff_easy} on this plastic set and that elements in array $\mathcal A_{[\ell]}$ are mutually unequal from condition (c), $\min \{\mathcal A_{[\ell]}\cup \{|v-w|: v\ne w\in \mathcal A_{[\ell]}\}\}= \Omega(\theta_\ell)$.
Since $\theta_\ell=\exp({\Omega(\log^{ba^\ell} m)})$ goes to infinity as $m$ goes to infinity, $\theta_{\ell+1}=\lfloor2^{\log^a \theta_\ell}\rfloor$ is reproducible for $\mathcal A_{[\ell]}$. Moreover, since $\ell <\tau$, $\mathcal A(U^{(\ell)})$ is not monotonic. From  (b) and Corollary~\ref{coro_order_fix}, $\mathcal A_{[\ell]}$ is also not monotonic. Together with (c), the feasibility of $\mathcal A_{[\ell]}$, both types of ancestors of $\mathcal A_{[\ell]}$ by $\theta_{\ell+1}$  exist. Hence property (1) is proved. Let $\mathcal A_{[\ell+1]}$ be any one of the ancestors. 
}

For property (2),  by (b) and Lemma~\ref{lemma_small_property}, neither of the largest elements of $\mathcal A_{[\ell]+}$ and $\mathcal A_{[\ell]-}$ is contained in the set $\{\theta_j: j\in [\ell]\}$. By Definition~\ref{def_ancestor}, for any $j\in [\ell]$, $\theta_j\in \mathcal A_{[\ell]+}$ implies  $\theta_j\in \mathcal A_{[\ell+1]+}$, while $\theta_j\in \mathcal A_{[\ell]-}$ implies  $\theta_j\in \mathcal A_{[\ell+1]-}$. By Remark~\ref{remark_of_link_lemma}, $\theta_{\ell+1}$ is a coefficient of $\mathcal A_{[\ell+1]}$. So property (2) holds for $\mathcal A_{[\ell+1]}$.

For property (3), again the fact that the set $\{\theta_j: j\in [\ell]\}$ is free from the largest elements in both sides leads to the commutative property of generating an ancestor and deleting $\{\theta_j: j\in [\ell]\}$ from $\mathcal A_{[\ell]}$. That is to say, no matter to generate which type of ancestors, say the $j$th type for any $j\in\{1, 2\}$, we have $(\mathcal A_{[\ell]})_j\backslash\{\theta_1, \ldots, \theta_\ell\}=(\mathcal A_{[\ell]}\backslash\{\theta_1, \ldots, \theta_\ell\})_j$. As a consequence, by Eq. (\ref{eq-scal}) in Lemma~\ref{le-appro}, for any $j\in\{1, 2\}$,
$$\mathcal A(U^{(\ell+1)})\approx (\mathcal A(U^{(\ell)}))_j\setminus\{\theta_{\ell+1}\}\overset{(b)}{\approx} (\mathcal A_{[\ell]}\backslash\{\theta_1, \ldots, \theta_\ell\})_j\backslash\{\theta_{\ell+1}\}=(\mathcal A_{[\ell]})_j\backslash\{\theta_1, \ldots, \theta_{\ell+1}\}.$$
The first approximation is with scale $\theta_{\ell+1}$, and naturally with $\theta_{1}$ since $\theta_{\ell+1}=o(\theta_1)$.
This means, for any $j\in\{1, 2\}$ with $\mathcal A_{[\ell+1]}=(\mathcal A_{[\ell]})_j$, $\mathcal A_{[\ell+1]}\backslash\{\theta_1, \ldots, \theta_{\ell+1}\}\approx \mathcal A(U^{(\ell+1)})$, and property (3) holds for $\mathcal A_{[\ell+1]}$.

Finally, property (4) holds for $\ell+1$  from Claim~\ref{claim_feasibility}.

\begin{claim}\label{claim_feasibility}
$\mathcal A_{[\ell+1]}$  is $(a, ba^{\ell+1})$-feasible with respect to $m$.
\end{claim}
\begin{proof} The proof is similar to that of Claim~\ref{claim_a1}.
 By Lemma~\ref{lemma_star_U_is_feasile}, $\mathcal A (U^{(\ell+1)})$ is $(a,b)$-feasible with respect to $m$,  so
  \begin{itemize}
    \item[(i')] elements in $\mathcal A (U^{(\ell+1)})$ are mutually unequal;
    \item[(ii')] the largest elements $\alpha\in \mathcal A(U^{(\ell+1)})_+$ and $\alpha' \in \mathcal A(U^{(\ell+1)})_-$ satisfy $|\alpha'- \alpha|=o(\alpha)$ and $\log \alpha= O(\log^a m)$;
    \item[(iii')] define $Z_1=Z(\mathcal A(U^{(\ell+1)}))$, i.e., $\mathcal A(U^{(\ell+1)})\cup\{|\alpha- \alpha'|\}\backslash\{\alpha, \alpha'\}$. Then for any $w\in Z_1$, $\Omega(\log^b m) = \log w =O(\log^a \alpha)$.
  \end{itemize}

   By  property (3) for $\ell+1$, $\mathcal A_{[\ell+1]}\backslash\{\theta_1, \ldots, \theta_{\ell+1}\}\overset{\theta_1}\approx \mathcal A (U^{(\ell+1)})$. Then by Corollary~\ref{coro_order_fix} and (i'), elements in $\mathcal A_{[\ell+1]}\backslash\{\theta_1, \ldots, \theta_{\ell+1}\}$ are mutually unequal.  On the other hand, $\theta_j, j\in [\ell+1]$ are also mutually unequal. From Lemma~\ref{lemma_small_property}, as $m$ is large enough, the largest element in the latter part, $\theta_1$, is smaller than the smallest element in the former part. So the elements in the union of them two, i.e., $\mathcal A_{[\ell+1]}$, are also mutually unequal. 

   Denote the corresponding coefficients of $\alpha$ and $\alpha'$ in $\mathcal A_{[\ell+1]}$ by the proven property (3) for $\ell+1$ as $\hat{\alpha}$ and $\hat{\alpha}'$, respectively. Then $\hat{\alpha}\in \mathcal A_{[\ell+1]+}$, $\hat{\alpha}'\in \mathcal A_{[\ell+1]-}$ are two largest elements in the corresponding part of $\mathcal A_{[\ell+1]}$, and $\hat{\alpha}= \alpha\pm O(\theta_1)$, $\hat{\alpha}'= \alpha'\pm O(\theta_1)$. Hence $|\hat{\alpha}'-\hat{\alpha}|= |\alpha' -\alpha| \pm O(\theta_1)=o(\alpha)=o(\hat{\alpha})$, and $\log\hat{\alpha}=\Theta(\log \alpha)=O(\log^a m)$.

Denote sets $Z= Z(\mathcal A_{[\ell+1]})= \mathcal A_{[\ell+1]}\cup\{|\hat{\alpha}-\hat{\alpha}'|\}\backslash \{\hat{\alpha}, \hat{\alpha}'\}$ and $Z_2= (\mathcal A_{[\ell+1]}\backslash\{\theta_1, \ldots, \theta_{\ell+1}\})\cup\{|\hat{\alpha}-\hat{\alpha}'|\}\backslash \{\hat{\alpha}, \hat{\alpha}'\}$. From the proven property (3) for $\ell+1$, for any element $\hat w\in Z_2$, there always exists a corresponding element $w\in Z_1$ such that $\hat w \approx w$ with scale $\theta_1$, and hence $\hat w=w\pm o(w)$.
From the $(a, b)$-feasibility of $\mathcal A(U^{(\ell+1)})$, $\Omega(\log^b m)=\log w= O(\log^a \alpha)=O(\log^a\hat\alpha)$. From $\hat w=w\pm o(w)$, $\Omega(\log^b m)=\log \hat w= O(\log^a\hat\alpha)$. For the left elements $\theta_j, j\in [\ell+1]$, it is trivial to check that $\Omega(\log^{ba^{\ell+1}} m)=\log \theta_j= O(\log^a\hat\alpha)$. Hence $\mathcal A_{[\ell+1]}$ is $(a, ba^{\ell+1})$-feasible.
\end{proof}

This completes the proof of  Theorem~\ref{thm_induc_link_lemma}.

\begin{remark}\label{remark_of_string_lemma}
From the proof of  Theorem~\ref{thm_induc_link_lemma}, the property (2)
 can be further improved as follows: As long as $\theta_i$ is set in one part of $\mathcal A_{[i]}$ ($\mathcal A_{[i]+}$ or $\mathcal A_{[i]-}$), for any $j\ge i$, $\theta_i$ is also located in the same part of $\mathcal A_{[j]}$.
\end{remark}

\section{From Monotonic Arrays to any $\mathcal A(U)$}\label{sec_sprine_lemma}

This section devotes to prove Theorem~\ref{thm_trans_main_thm} by using the link lemma. Let us first recall the theorem.
\begin{thmbis}{thm_trans_main_thm}
For a fixed integer $t$ and reals $a,b$ satisfying $0<b\le a^t<1$, when $m$ is large enough,  let $R=\{b_1, \ldots, b_t\}$ be an $(a, b)$-plastic subset with respect to $m$.
Then for any $k$-permutation sequence $U$ of $R$ for some $k\in[3, t]$, there exists a constant $c=c(\tau(U), a)\in (0, 1)$ such that an $\mathcal L(U)$-free set $M\subset [m]$ exists with $|M|=\frac m {2^{O(\log^{c} m)}}$.
\end{thmbis}

As we have mentioned in Section~\ref{sec_notations}, we rephrase this theorem in the language of feasibility and goodness. For any $(a, b)$-plastic set $R(m)$ of some large enough $m$, we only need to show that for any $k$-permutation sequence $U$ of $R$ for some $3\le k\le t$, the corresponding bipartite $(a, b)$-feasible array $\mathcal A(U)$ with respect to $m$ is $c$-good under location $m$.

Let $\mathcal{C}$ be the set of all monotonic invariant arrays with respect to $m$ such that all coefficients are in $[2^{O(\log^a m)}]$, where $a$ is from Theorem~\ref{thm_trans_main_thm}. Then by Lemma~\ref{corollary_one_side_along}, all arrays in $\mathcal{C}$ are $c_0$-good under location $m$ with $c_0=\frac{1+a}2\in(0, 1)$. Let $R(m)$ be an $(a, b)$-plastic set of size $t$, and let $U$ be a $k$-permutation sequence of $R$ for some $k\in[3, t]$. In Algorithm~\ref{alg_1}, we show a way to generate from $\mathcal A(U)$ a string of ancestors  $\mathcal A_{[0]}, \mathcal A_{[1]}, \ldots, \mathcal A_{[\tau]}$, which ends with an array in $\mathcal{C}$, see Lemma~\ref{string_lemma_3}. Since each array in $\mathcal{C}$ is $c_0$-good, and each $\mathcal A_{[i]}$ is  $(a, ba^i)$-feasible with respect to $m$ from Theorem~\ref{thm_induc_link_lemma},  we can apply the link lemma to each ``link'' through the string to deduce the goodness of $\mathcal A_{[0]}=\mathcal A(U)$. 


\begin{algorithm}
\begin{algorithmic}
\STATE {\bf Input:} a permutation sequence $U$
\STATE compute $\tau(U)$ and $\chi(U)$
\STATE set $\mathcal A_{[0]}=\mathcal A(U)$; set $\theta_1=\lfloor2^{\log^a(\min U)}\rfloor$
\FORALL {$i\in [\tau]$}
\IF {$\chi(U)_i =0$}
\STATE set $\mathcal A_{[i]}$ as the first type ancestor of $\mathcal A_{[i-1]}$ with $\theta=\theta_i$
\ELSE
\STATE set $\mathcal A_{[i]}$ as the second type ancestor of $\mathcal A_{[i-1]}$ with $\theta=\theta_i$
\ENDIF
\STATE set $\theta_{i+1}=\lfloor2^{\log^a \theta_i}\rfloor$
\ENDFOR
\end{algorithmic}
\caption{Finding a string from array $\mathcal A(U)$ to $\mathcal{C}$}
\label{alg_1}
\end{algorithm}

\begin{lemma}\label{string_lemma_3}
In Algorithm~\ref{alg_1}, $\mathcal A_{[\tau]}$ is always in $\mathcal{C}$.
\end{lemma}
\begin{proof}
Let us first recall the notations in Section~\ref{subsec_cyc_diff_and_ter_num}. Denote $U^{(i)}=(u_1^{(i)}, u_2^{(i)}, \ldots, u_{k-i}^{(i)})$ for any $i\in [\tau]$, and denote $\max(i)$ as the subindex of the maximum element in $U^{(i)}$ under $\mathbb Z_{k-i}$.

Consider the case when $U^{(\tau)}$ is increasing. By the definition, $\mathcal L(U^{(\tau)})$ is monotonic and $\mathcal A(U^{(\tau)})$ is of type $(k-\tau-1, 1)$. To show that $\mathcal A_{[\tau]}$ is  in $\mathcal{C}$, it suffices to show that the newly added coefficient $\theta_i$ is added to $\mathcal A_{[i]+}$ in each step from $\mathcal A_{[i-1]}$ to $\mathcal A_{[i]}$. Notice that the largest two elements in $\mathcal A(U^{(i-1)})_+$ and $\mathcal A(U^{(i-1)})_-$ are $\alpha= u_{\max(i-1)}^{(i-1)}- u_{\max(i-1)-1}^{(i-1)}$ and $\alpha'= u_{\max(i-1)}^{(i-1)}- u_{\max(i-1)+1}^{(i-1)}$, respectively. Denote the corresponding elements of $\alpha$ and $\alpha'$ in $\mathcal A_{[i-1]}$ as $\hat{\alpha}$ and $\hat{\alpha}'$, respectively. From Corollary~\ref{coro_order_fix} and Theorem~\ref{thm_induc_link_lemma} (3), $\hat{\alpha}$ is the largest in $\mathcal A_{[i-1]+}$, and $\hat{\alpha}'$ is the largest in $\mathcal A_{[i-1]-}$. Moreover, the order between $\hat\alpha$ and $\hat\alpha'$ is the same as $\alpha$ and $\alpha'$.

If $\chi(U)_i= 0$, by the definition of the character, $u_{\max(i-1)-1}^{(i-1)}< u_{\max(i-1)+1}^{(i-1)}$, and hence $\hat{\alpha}'-\hat{\alpha}<0$. Then Algorithm~\ref{alg_1} will choose $\mathcal A_{[i]}$ as the first type ancestor of $\mathcal A_{[i-1]}$, which puts $\theta_i$ in $\mathcal A_{[i]+}$ exactly in this case. If $\chi(U)_i= 1$, then $u_{\max(i-1)-1}^{(i-1)}> u_{\max(i-1)+1}^{(i-1)}$, and hence $\hat{\alpha}'-\hat{\alpha}>0$. Then Algorithm~\ref{alg_1} will choose $\mathcal A_{[i]}$ as the second type ancestor of $\mathcal A_{[i-1]}$, which also puts $\theta_i$ in $\mathcal A_{[i]+}$.

The above analysis shows that $\mathcal A_{[\tau]}\backslash\{\theta_1, \ldots, \theta_{\tau}\}$ is of $(k-\tau-1, 1)$ type, while $\theta_1, \ldots, \theta_{\tau}$ are all located in $\mathcal A_{[i]+}$ (see Remark~\ref{remark_of_string_lemma}). Thus, $\mathcal A_{[\tau]}$ is of $(k-1, 1)$ type, which is monotonic.

Consider the case when $U^{(\tau)}$ is decreasing. The argument is quite similar. The equation $\mathcal A(U^{(\tau)})$ is of type $(1, k-\tau-1)$, and we need to ensure that $\theta_i$ is added to $\mathcal A_{[i]-}$ in each step from $\mathcal A_{[i-1]}$ to $\mathcal A_{[i]}$. The expressions of two largest elements $\alpha$ and $\alpha'$ in $\mathcal A(U^{(i-1)})_+$ and $\mathcal A(U^{(i-1)})_-$, and the corresponding elements $\hat{\alpha}$ and $\hat{\alpha}'$  in $\mathcal A_{[i-1]}$ are similarly obtained. For the same reason, $\hat{\alpha}$ is the largest in $\mathcal A_{[i-1]+}$, and $\hat{\alpha}'$ is the largest in $\mathcal A_{[i-1]-}$.

If $\chi(U)_i= 0$, then $u_{\max(i-1)-1}^{(i-1)}> u_{\max(i-1)+1}^{(i-1)}$, and hence $\hat{\alpha}'-\hat{\alpha}>0$. Then Algorithm~\ref{alg_1} will choose $\mathcal A_{[i]}$ as the first type ancestor of $\mathcal A_{[i-1]}$, which puts $\theta_i$ in $\mathcal A_{[i]-}$. If $\chi(U)_i= 1$, then $u_{\max(i-1)-1}^{(i-1)}< u_{\max(i-1)+1}^{(i-1)}$, and hence $\hat{\alpha}'-\hat{\alpha}<0$. Then Algorithm~\ref{alg_1} will choose $\mathcal A_{[i]}$ as the second type ancestor of $\mathcal A_{[i-1]}$, which also puts $\theta_i$ in $\mathcal A_{[i]-}$.

The above analysis shows that $\mathcal A_{[\tau]}\backslash\{\theta_1, \ldots, \theta_{\tau}\}$ is of $(1, k-\tau-1)$ type, while $\theta_1, \ldots, \theta_{\tau}$ are all located in $\mathcal A_{[i]-}$. Thus, $\mathcal A_{[\tau]}$ is of $(1, k-1)$ type, which is monotonic.
\end{proof}

Now we prove Theorem~\ref{thm_trans_main_thm}.

\begin{proof}[Proof of Theorem~\ref{thm_trans_main_thm}]
If $\tau(U)=0$, things are trivial by Lemma~\ref{corollary_one_side_along} with $c=(1+a)\slash2$. Otherwise,
do Algorithm~\ref{alg_1} on $\mathcal A(U)$ and we get a string of arrays $\mathcal A(U)=\mathcal A_{[0]}$, $\mathcal A_{[1]}, \ldots, \mathcal A_{[\tau]}$ such that $\mathcal A_{[i]}$  is an ancestor of $\mathcal A_{[i-1]}$ by $\theta_i$ for each $i\in [\tau]$, and $\mathcal A_{[\tau]}$ is monotonic by Lemma~\ref{string_lemma_3}. 

For any $i\in [0, \tau]$, define $\alpha_i$ as the small one between the two largest elements in $\mathcal A_{[i]+}$ and $\mathcal A_{[i]-}$, respectively; define $\sigma_i$ as the sum of elements in $\mathcal A_{[i]+}$. For any $i\in [\tau]$, define $m_i=\lfloor\alpha_{i-1}\slash\sigma_i\rfloor$.

By  Theorem~\ref{thm_induc_link_lemma} (3), $\alpha_{i-1}$ approximates the largest element in $\mathcal A(U^{(i-1)})$, and hence it approximates the largest element in $U^{(i-1)}$. Moreover, for any element $v\in \mathcal A_{[i]}$, by the $(a, b)$-plasticity of $R$, $\log v=O(\log^a\alpha_{i-1})$, which leads to $\log \sigma_i = O(\log^a\alpha_{i-1})$. So $m_i$ goes to infinity with $m$, and $\log\alpha_i= O(\log^a m_i)$ for each $i\in [\tau]$.

We prove by induction on $j\in [0,\tau]$ that each $\mathcal A_{[\tau-j]}$ is $c_j$-good under location $m_{\tau-j}$ for some constant $c_j$. For $j=0$, $\mathcal A_{[\tau]} \in \mathcal{C}$, i.e.,  $\mathcal A_{[\tau]}$ is monotonic by Lemma~\ref{string_lemma_3}. From the analysis above, for any $v\in\mathcal A_{[\tau]}$, $v\in [2^{O(\log^a m_{\tau})}]$. By Lemma~\ref{corollary_one_side_along}, $\mathcal A_{[\tau]}$ is $c_0=\frac{1+a}2$-good under location $m_{\tau}=\lfloor\alpha_{\tau-1}\slash \sigma_{\tau}\rfloor$.

For $j=1$, consider the link from $\mathcal A_{[\tau-1]}$ to $\mathcal A_{[\tau]}$.  From Theorem~\ref{thm_induc_link_lemma} (4), $\mathcal A_{[\tau-1]}$ is $(a, ba^{\tau-1})$-feasible.  Applying Lemma~\ref{lemma_linkage}, i.e., the link lemma, by letting $ba^{\tau-1}$ as the new $b$, $\mathcal A_{[\tau-1]}$ is $c_1$-good under location $m_{\tau-1}$, where $c_1= 1-a+a\max\{a, c_0\}\in (0, 1)$ by Remark~\ref{remark_of_link_lemma}.

Iteratively define $c_i=1-a+a \max\{a, c_{i-1}\}$ for any $i\in [2,\tau]$. Hence $c_{\tau}$ is determined only by parameters $\tau$ and $a$. Suppose we have proved that for any $j<\ell$ with $\ell \in [2,\tau-1]$, $\mathcal A_{[\tau-j]}$ as an $(a, ba^{\tau-j})$-feasible array with respect to $m$ is $c_j$-good under location $m_{\tau-j}$. Consider $\mathcal A_{[\tau-\ell]}$.

From Theorem~\ref{thm_induc_link_lemma} (4), $\mathcal A_{[\tau-\ell]}$ is $(a, ba^{\tau-\ell})$-feasible. Moreover, $\mathcal A_{[\tau-\ell+1]}$ is $c_{\ell-1}$-good under location $m_{\tau-\ell+1}=\lfloor\alpha_{\tau-\ell}\slash\sigma_{\tau-\ell+1}\rfloor$ by the induction hypothesis. So it is feasible to use the link lemma from $\mathcal A_{[\tau-\ell]}$ to $\mathcal A_{[\tau-\ell+1]}$. This means for any location $m'$ satisfying $\log \alpha_{\tau-\ell}=O(\log^a m')$, $\mathcal A_{[\tau-\ell]}$ is $c_\ell$-good under location $m'$. Since $\log \alpha_{\tau-\ell}=O(\log^a m_{\tau-\ell})$ when $\ell< \tau$, $\mathcal A_{[\tau-\ell]}$ is $c_{\ell}$-good under location $m_{\tau-\ell}$.


As a result, the $c_\ell$-goodness of $\mathcal A_{[\tau-\ell]}$ under location $m_{\tau-\ell}$ is right until $\ell=\tau-1$. For $\ell=\tau-1$, we have $m_{\tau-\ell}=m_1=\Theta(\max U\slash \sigma_1)$. By the $(a, b)$-plasticity of $R$, 
$\log \alpha_0=\Theta(\log (\max U))=O(\log^a m)$.
From the link lemma, $\mathcal A(U)=\mathcal A_{[0]}$ is $c_{\tau}$-good under location $m$, which means that by choosing $c=c_{\tau}$, there always exists an $\mathcal L(U)$-free set $M\subset [m]$ such that $|M|=\frac m {2^O(\log^c m)}$.
\end{proof}



\section{Concluding Remarks}\label{sec_conclusion}
This paper gives a full positive answer to Question~\ref{question_blackburn} when $N\in [u, 2u-3]$. To see this for any $N\in [u, 2u-3]$, on one hand, if $t\ge 3$ or $\min\{w_1, w_2\}>1$, from Theorem~\ref{thm_all} and the monotonicity of $C(N, q, \{w_1, \ldots, w_t\})$, $$C(N, q, \{w_1, \ldots, w_t\})\ge C(u, q, \{w_1, \ldots, w_t\})> q^{2-o(1)}.$$
{On the other hand, if $t=2$ and $\{w_1, w_2\}=\{1, w\}$ for some $w$, from the Johnson-type bound~\cite{shangguan2016separating} and the Reed-Solomn code construction~\cite{blackburn2003frameproof}, $C(N, q, \{1, w\})=\Theta(q^2)$. It is interesting to  study  Question~\ref{question_blackburn} for more $N$.

We have proved that for any given $N, t$ and type $\{w_1, \ldots, w_t\}$ with $N\in [u, 2u-3]$, either $C(N, q, \{w_1, \ldots, w_t\})=q^{2-o(1)}$ or $\Theta(q^2)$. However, it is  determined whether $C(N, q, \{w_1, \ldots, w_t\})=q^{2-o(1)}$ or $\Theta(q^2)$ completely only when $N=u$. It is also interesting to study on the following problem.
\begin{problem}\label{problem_1}
Determine whether $C(N, q, \{w_1, \ldots, w_t\})=q^{2-o(1)}$ or $\Theta(q^2)$ completely for any given $N, t$ and type $\{w_1, \ldots, w_t\}$ with $N\in [u+1, 2u-3]$.
\end{problem}
Results of Problem~\ref{problem_1} are seldom. Even for the case $t=2$ there is no  complete answer. When $\min\{w_1, w_2\}=1$, it has been mentioned that $C(N, q, \{1, w\})=\Theta(q^2)$ for any $N\in [u, 2u-3]$. When $\min\{w_1, w_2\}=2$, we can use the same method as in \cite{blackburn2003frameproof} to get $C(N, q, \{2, u-2\})=\Theta(q^2)$ with $N=2u-3$ for all fixed $u\ge 4$. The first unsolved case for $t=2$ is $C(6, q, \{2, 3\})$. As a probable starting point for solving Problem~\ref{problem_1}, one can try to solve all $t=2$ cases. For the convenience of readers, we list  below  all cases of $C(N, q, \{2, u-2\})$ with  $N\in [u, 2u-3]$ for $u\geq 4$, in which each $C(N, q, \{w_1, w_2\})$ is abbreviated as $(N; \{w_1, w_2\})$.
}

\begin{table}[h]
\centering
\begin{tabular}{cccccc}
$\mathbf{(4; \{2, 2\})}$& $\mathbf{(5; \{2, 2\})}$& & & & \\
$\mathbf{(5; \{2, 3\})}$& (6; \{2, 3\})& $\mathbf{(7; \{2, 3\})}$& & & \\
$\mathbf{(6; \{2, 4\})}$& (7; \{2, 4\})& (8; \{2, 4\})& $\mathbf{(9; \{2, 4\})}$& & \\
$\mathbf{(7; \{2, 5\})}$& (8; \{2, 5\})& (9; \{2, 5\})& (10; \{2, 5\}) & $\mathbf{(11; \{2, 5\})}$& \\
\vdots &     \vdots &            \vdots &        \vdots &   $\cdots$    & $\ddots$
\end{tabular}
\caption{Cases of $C(N, q, \{2, w\})$}
\end{table}

In the table, the first column is the case when $N=u$ and hence the values are all $o(q^2)$; in each row, all right-most ones are the case when
 $N=2u-3$ and hence the values are all $\Theta(q^2)$.
From the monotonicity, $C(N, q, \{w_1, w_2\})\le C(N+1, q, \{w_1, w_2\})$. So values are increasing from left to right in each row. From the Johnson-type bound \cite{shangguan2016separating}, $C(N, q, \{w_1, w_2\})\le q+ C(N-1, q, \{w_1, w_2-1\})$. So for each column in the table the orders are non-increasing from top to the bottom.

Therefore, if one can prove that one member $(N; \{2, w\})$ is of size $o(q^2)$, all its lower left part in the table are all of sizes $o(q^2)$; if one can prove that one member $(N; \{2, w\})$ is of size $\Theta(q^2)$, all its upper right part in the table are all of sizes $\Theta(q^2)$.



\vskip 10pt

\end{document}